\documentclass[reqno,a4paper,11pt]{amsart}
\usepackage{mathtools} 
\usepackage{amsmath,amsthm,amsbsy,amssymb, comment}

\theoremstyle{plain}
\newtheorem{theorem}{Theorem}

\newtheorem{lemma}{Lemma}
\newtheorem{proposition}{Proposition}

\newcommand{\cA}{\mathcal{A}}

\renewcommand{\Re}{\operatorname{Re}}

\newcommand{\NN}{\mathbb{N}}

\theoremstyle{definition}

\theoremstyle{remark}

\usepackage[font={small}]{caption}

\usepackage{subfigure}  
\usepackage[pdftex]{graphicx}
\graphicspath{{IMAGES/}{./}}
\usepackage[backref=page]{hyperref} 
\begin{document}
\title[A bias parity question for Sturmian words]
{A bias parity question for Sturmian words}

\author[Cristian Cobeli]{Cristian Cobeli}
\address{"Simion Stoilow" Institute of Mathematics of the Romanian Academy, 21 Calea Grivitei 
Street, P. O. Box 1-764, Bucharest 014700, Romania}
\email{cristian.cobeli@imar.ro}

\author{Alexandru Zaharescu}
\address[Alexandru Zaharescu]{Department of Mathematics, University of Illinois, 1409 West Green 
Street, Urbana, IL 61801, USA.}

\address[Alexandru Zaharescu]{"Simion Stoilow" Institute of Mathematics of the Romanian Academy, 21 
Calea Grivitei 
Street, P. O. Box 1-764, Bucharest 014700, Romania}
\email{zaharesc@illinois.edu}



\dedicatory{This  article$^*$ is dedicated to Professor Solomon Marcus %
(1925 - 2016).
}

\begin{abstract}
We analyze a natural parity problem on the divisors associated to Sturmian words. We find a clear 
bias towards the odd divisors and obtain a sharp asymptotic estimate for the average of 
the difference odd-even function tamed by a mollifier, which proves various experimental results.
\end{abstract}
\maketitle

\renewcommand*{\thefootnote}{\fnsymbol{footnote}}
\footnotetext[1]{A preliminary version of this paper was initiated in 2011 by the second author, 
inspired by the articles Marcus~\cite{M2004} and Marcus and Monteil~\cite{MM2006} and a fruitful 
discussion with Professor Solomon Marcus, in 2010, at the Faculty of Mathematics of University of 
Bucharest.}
\renewcommand*{\thefootnote}{\arabic{footnote}}
\setcounter{footnote}{0}

\section{Introduction}

Sturmian words are remarkable objects that lie at the frontier between order and chaos in the ample 
comprising world of binary sequences.
They have interesting properties of which a few characterize them throughly. 
One of them says that a Sturmian word is a binary sequence that is not ultimately periodic and has 
minimal 
complexity, that is, it contains exactly $n+1$ distinct blocks of $n$ consecutive letters for each 
$n\ge 0.$ 
Sturmian words where first introduced in 1940 by Morse and Hedlund~\cite{MH1940} and since then, 
they have become a topic of intensive research~\cite{HKW2018}, \cite{K2013}, \cite{LR2007}, 
\cite{L2011}, 
\cite{LL2015}, \cite{M2004}, \cite{MM2006},  \cite{M2000}, \cite{OEIS},
\cite{PW2017},
\cite{RV2017}.

Since an infinite word formed by just two letters is given by the sequence of natural numbers 
indicating the 
positions of just one of the two letters, a natural question is if this sequence has any special 
properties 
in the case of Sturmian words. Thus our object is to investigate the following parity problem that 
has an 
arithmetical flavor.

Suppose $w=w_1w_2\cdots$ is a binary word with letters from the alphabet $\cA=\{a,b\}$. 
For each $n>0$, denote by $o_w(n)$ the number of divisors $j\mid n$ for which  $w_j=b$ and the 
multiplicative-complement $n/j$ is odd, and similarly, let $e_w(n)$ be the number of divisors $j\mid 
n$ for 
which  $w_j=b$ and its multiplicative-complement $n/j$ is even.
Thus
\begin{equation*}
	\begin{split}
	o_w(n)&:=|\{j\in\NN : j \text{ divides } n, w_j=b, n/j \text{ odd } \}|,\\
	e_w(n)&:=|\{j\in\NN : j \text{ divides } n, w_j=b, n/j \text{ even } \}|\,.
	\end{split}
\end{equation*}
The behavior of the parity functions is quite irregular, as can be seen in 
Figure~\ref{FigFibonacciOddEven-2-300} drawn for the Fibonacci word $w=abaababaaba\dots$, a 
generic, recursively generated Sturmian word~\cite{B1986}, \cite{M2009}, 
\mbox{\cite[A003849]{OEIS}}.

We measure the deviation from the equilibrium of parity by the difference
\begin{equation}\label{defD}
	\begin{split}
	D_w(n) = o_w(n)-e_w(n)\,.
	\end{split}
\end{equation}
The question is whether, for an arbitrary Sturmian word, there exists a bias in the distribution 
towards evens or odds and if the difference function is accurate enough to weight any possible 
dependence on the word $w$.

\advance\leftmargini -2.3em
\begin{quote}
     \textbf{Question 1.} \textsl{For any arbitrary Sturmian word $w$, is there a particular 
tendency of 
$D_w(n)$ of 
being more positive than negative, or conversely, as $n$ increases towards infinity?} 
\end{quote}

\begin{figure}[ht]
 \centering
 \mbox{
 \subfigure{
    \includegraphics[width=0.48\textwidth]{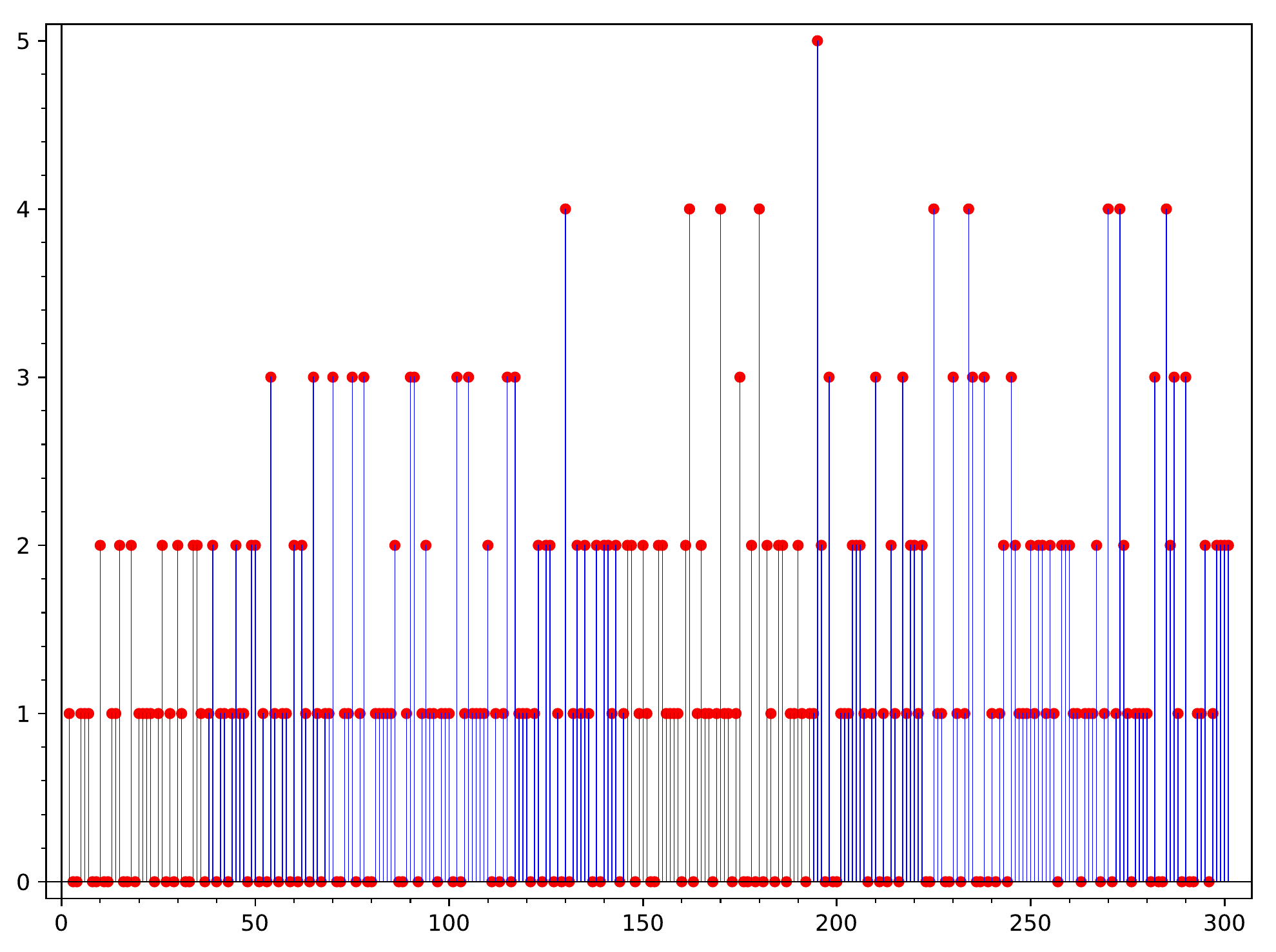}
 }
 \subfigure{
    \includegraphics[width=0.48\textwidth]{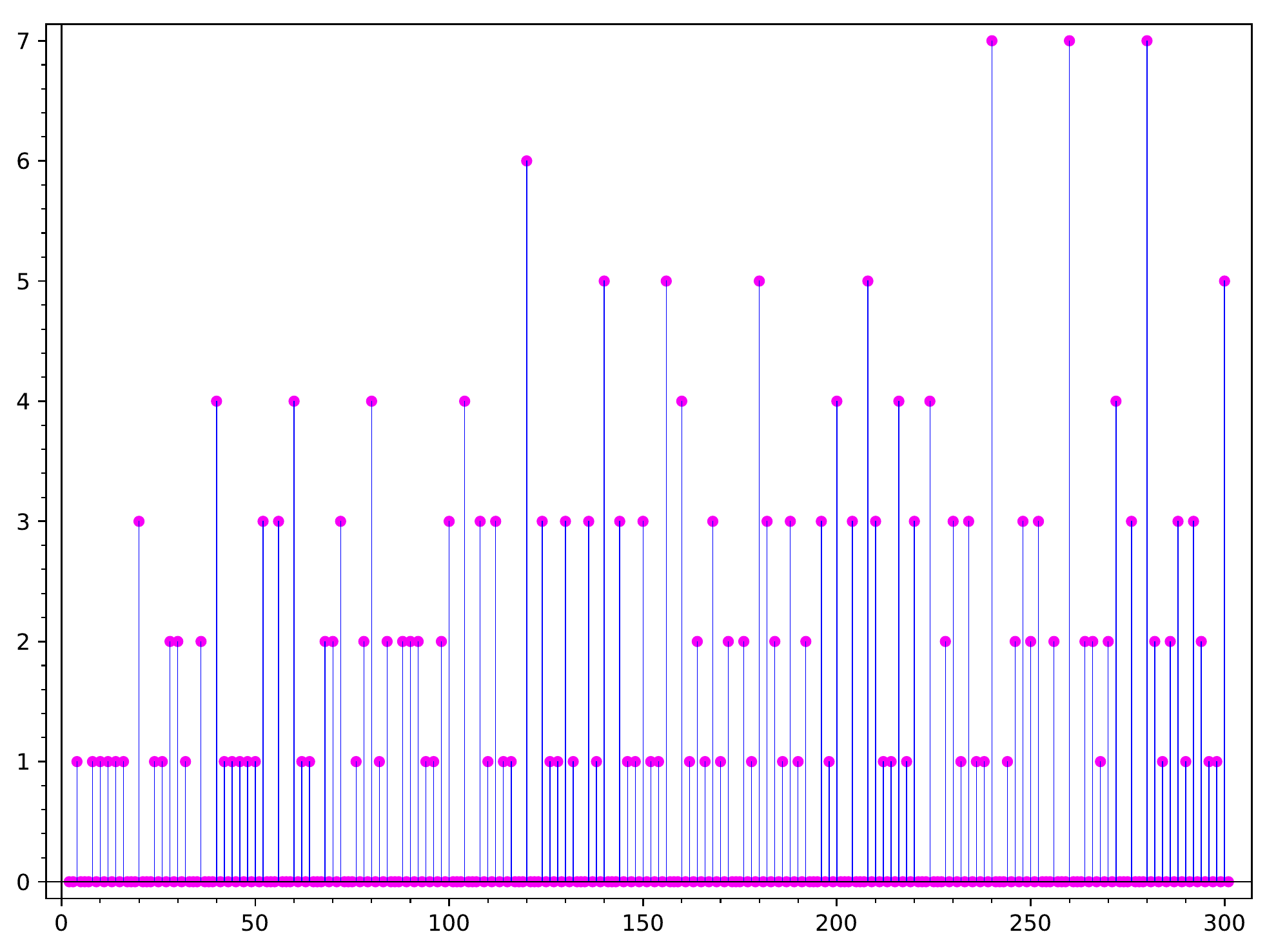}
 }
 }
\caption{The Fibonacci parity functions $o_w(n)$ (left) and $e_w(n)$ (right) for 
$n\in[2,300]$.
}
 \label{FigFibonacciOddEven-2-300}
 \end{figure}

At first look, while perhaps for low values of $n$, the function~\eqref{defD} that counts the 
deviation between the number of even or odd divisors may take somewhat disparate values, perhaps in 
the long run or at least on average, we should not expect any bias on the positive or negative 
values over the normal statistical variance.
At least that happens in similar questions, such as the Lehmer problem~\cite{CZ2001} or on the 
parity of pairs of residue classes and their inverses (\cite{CGZ2003}, 
\cite{CVZ2003}, \cite{CVZ2012}, \cite{CZ2002}, \cite{CZ2006}). Although the balance is most likely 
tilted in the 
case of primes, where up to some limit, are more preponderant those of the form $4k + 3$ than those 
of the form $4k + 1$ (Chebyshev's bias, see Rubinstein and 
Sarnak~\cite{RS94}) or the more general Shanks-R\'enyi prime races problem 
(see Lamzouriz~\cite{L12}).

In the case of the Fibonacci word, the distribution of $D_w(n)$ looks still quite random, but if we 
calculate its average, 
$\sum_{n=1}^x D_w(n)$,
we observe a clear tendency of linear increase (see 
Figure~\ref{FigFibonacci2-1000}). This shows a strong bias towards the odd divisors. Is there an 
explanation for such a strong discrepancy. 
We will see that the same behavior is characteristic and does not depend of the Fibonacci word $w$ 
and in Section~\ref{SectionAllPositiveIntegers} we will see that the appearance of exactly that 
slope of increase of the average of $D_w(n)$ is quite natural.
\begin{figure}[ht]
 \centering
 \mbox{
 \subfigure{
    \includegraphics[width=0.48\textwidth]{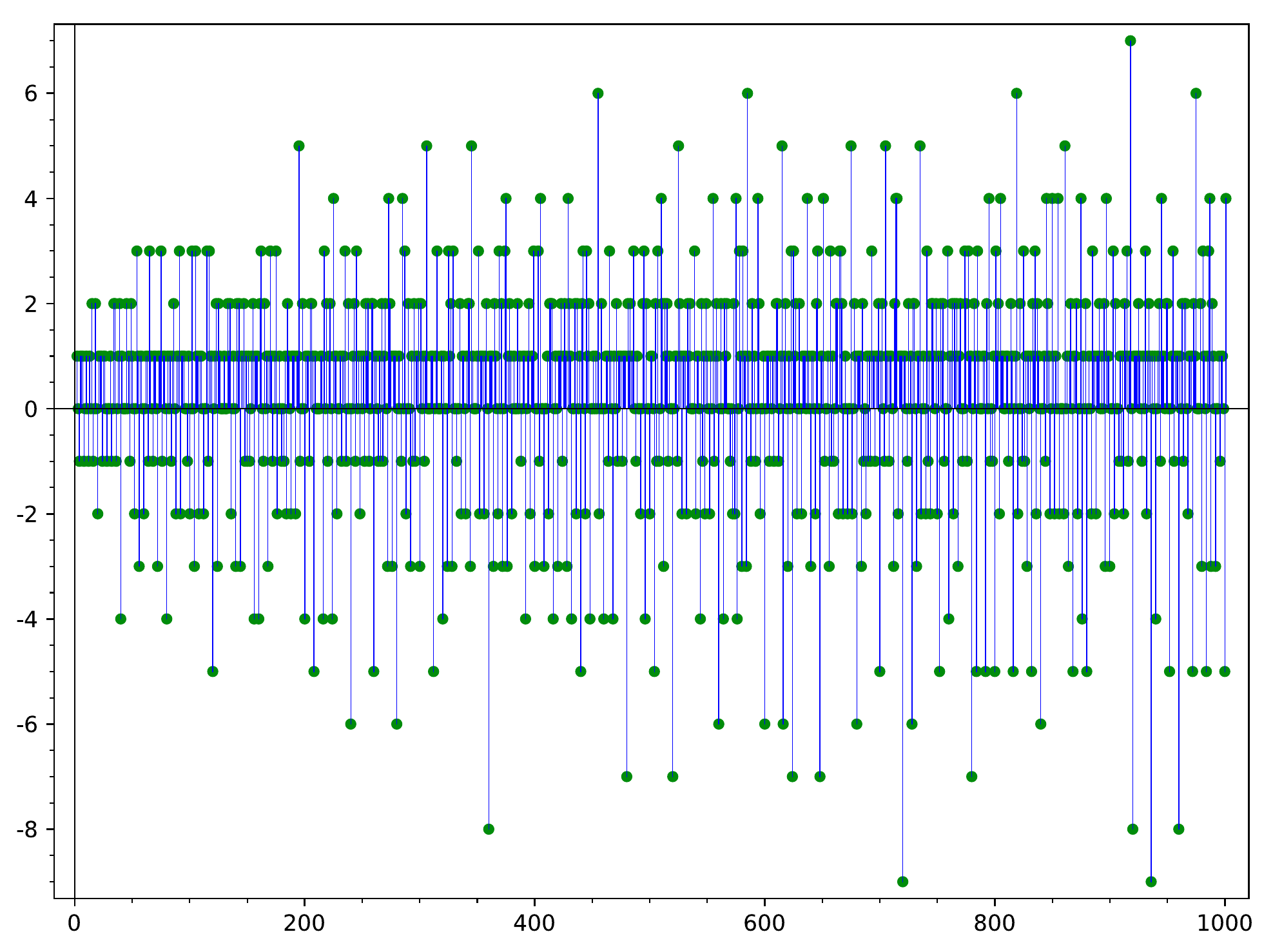}
 }
 \subfigure{
    \includegraphics[width=0.48\textwidth]{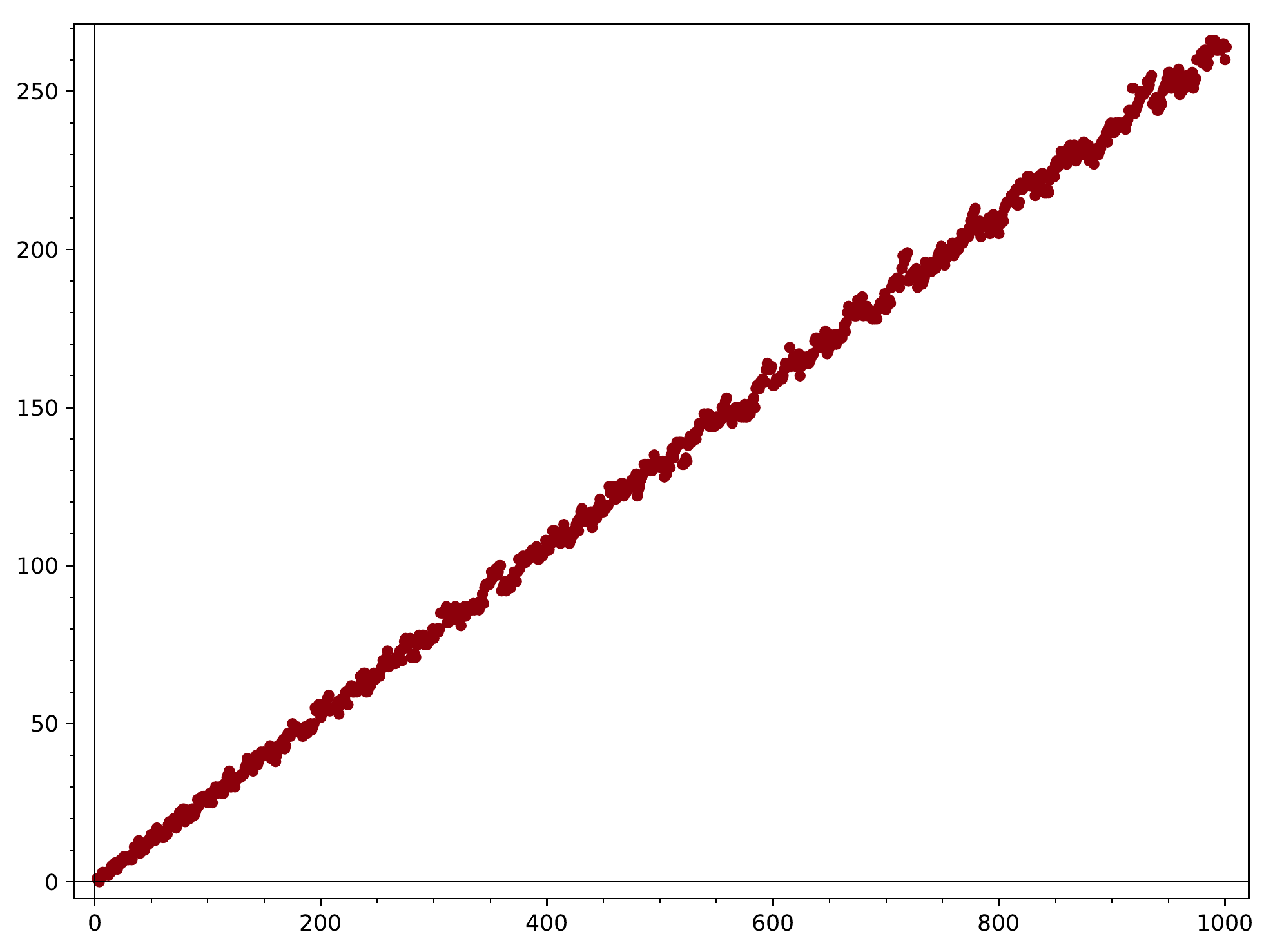}
 }
 }
\caption{The spread of 	$D_w(n)$ (left) and its partial averages $\sum_{n=1}^x D_w(n)$ for 
$x\in [2, 1000]$ (right).	
In the images, $w$ is the Fibonacci word.
}
 \label{FigFibonacci2-1000}
 \end{figure}

For any Sturmian word $w$ over the alphabet $A = \{ a, b\}$, we denote by $\beta_w$ the limit 
proportion of the occurrence of letter $b$, that is,
\begin{equation*}
\beta_w:= \lim_{n\rightarrow \infty} \frac {|\{ 1\le j\le n :  w(j) = b\}|}{n}\,.
\end{equation*}
The existence of $\beta_w$ is assured for any Sturmian word~\cite[Chapter 9]{AS2003}. For example, 
for the Fibonacci word, its precise value is $(3-\sqrt{5})/2$.

Our treatment of Question 1 conveniently tames the partial averages of the difference function  
by a mollifier. Thus, we wish to evaluate
\begin{equation}\label{e14}
M_w(x) = \sum_{n=1}^x D_w(n) \left(1-\frac nx\right).
\end{equation}
In particular, we would like to answer to the following question:
does $M_w(x)$ have constant sign for $x$ large enough, or does it have infinitely many
changes of sign?  As we shall see below, $M_w(x)$ is positive for $x$ large enough. We 
will actually prove a stronger statement, obtaining a sharp asymptotic formula for $M_w(x)$.
\begin{theorem}\label{Theorem1} Let $w$ be a Sturmian word.
Then, for
any $\delta>0$,
\begin{equation}\label{eq24}
M_w(x) = \frac{\beta_w\log 2}{2}\,x + O_\delta\left(x^{\frac13+\delta}\right)\,.
\end{equation}
\end{theorem}

Let us remark that since $\beta_w>0$, the asymptotic estimation~\ref{eq24} implies that $M_w(x)>0$ 
for all sufficiently large $x$, proving a strong bias towards the odd divisors for all Sturmian 
words.
We also mention that the Fibonacci word, the slope of the main term of the estimate~\eqref{eq24} is
$\frac{(3-\sqrt{5})\log 2}{4}\approx 0.264758$.
\medskip

The paper is organized as follows: in Section~\ref{SectionAllPositiveIntegers} we discuss the 
analogue question on the parity of the divisors of positive integers instead of Sturmian words, in 
the next two sections we employ the techniques developed for the Dirichlet series associated to 
Sturmian words and then use them to prove Theorem~\ref{Theorem1} in the last section.

\section{The parity problems for the sequence of positive 
integers}\label{SectionAllPositiveIntegers}
In this section we look on the parity problem for the infinite word  $u=bbb\dots$, whose letters 
are all equal, with $b$. The word $u$ is not a Sturmian word, but its analysis puts on 
perspective the more complex cases.

In the following, for simplicity, we will drop the subscript and 
write 
$o(n), e(n)$ and $D(n)$ instead of 
$o_u(n), e_u(n)$ and $D_u(n)$.

Like in the case of Sturmian words, the values of these parity functions  are very irregular (see 
Figure~\ref{FigOddEven-2-300}) and the involvement of the primes is part of the motive.

\begin{figure}[ht]
 \centering
 \mbox{
 \subfigure{
    \includegraphics[width=0.48\textwidth]{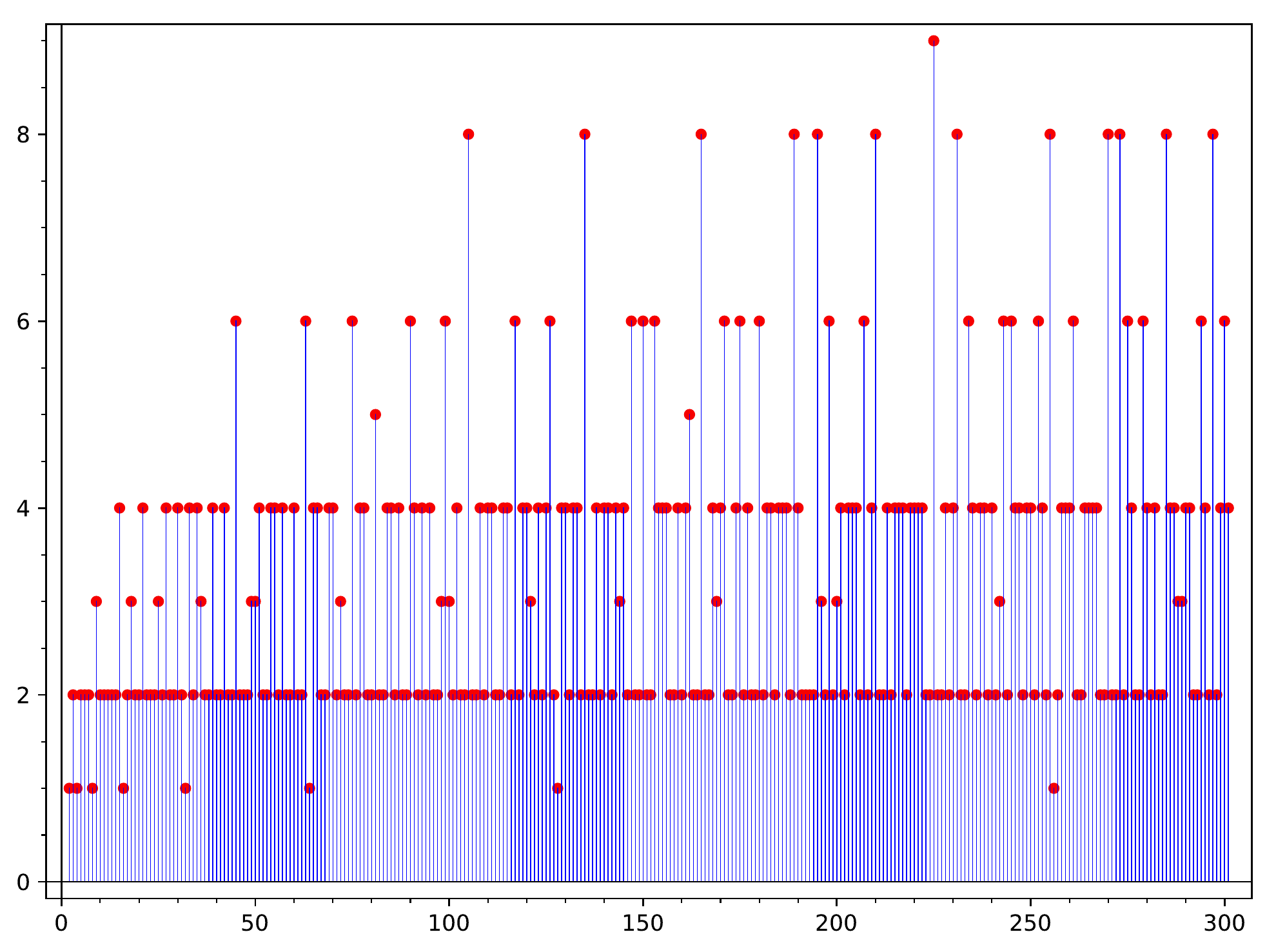}
 }
 \subfigure{
    \includegraphics[width=0.48\textwidth]{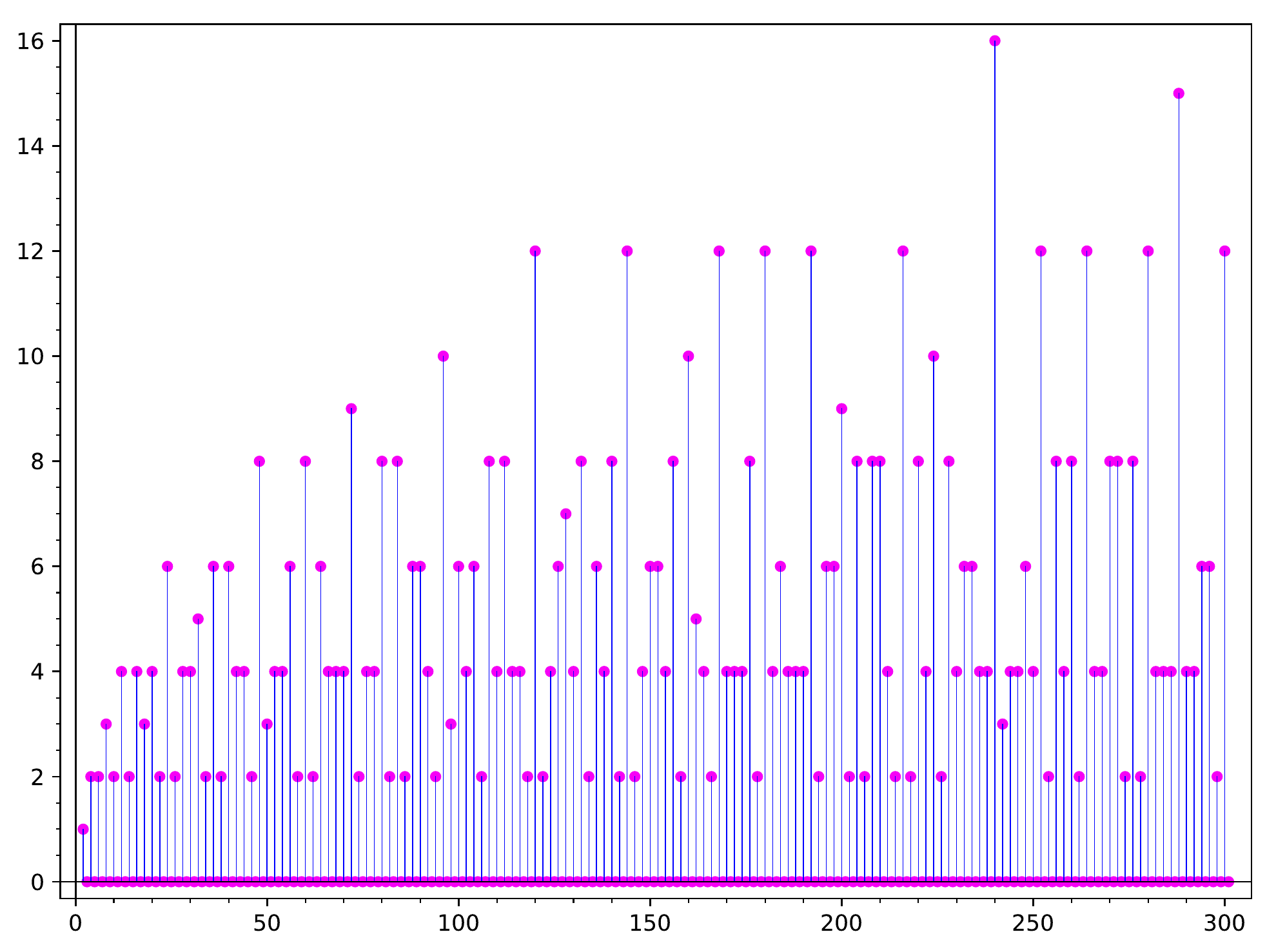}
 }
 }
\caption{The values of 	$o(n)$ (left) and $e(n)$ (right) for $n\in[2,300]$.
}
 \label{FigOddEven-2-300}
 \end{figure}

For any $n\ge 1$, the parity functions  $o(n), e(n)$  can easily be obtained if the decomposition in 
the prime factors of $n$ is known. Indeed, let $n=2^\alpha r$, with $\alpha>0$, $r$ odd and
$r=p_1^{\alpha_1} \cdots p_k^{\alpha_k}$. Then, all the divisors of $n$ are the terms of the 
sum obtained after all multiplications are done in the formal product
\[
 (1+2+\cdots+2^\alpha)(1+p_1+\cdots+p_1^{\alpha_1})\cdots (1+p_k+\cdots+p_k^{\alpha_k})\,.
\]
In particular, we see that the total number of divisors of $n$ is equal to 
$(\alpha+1)(\alpha_1+1)\cdots(\alpha_k+1)=(\alpha+1)d(r)$, where $d(\cdot)$ denotes the number of 
divisors function.

Now, if $\alpha=0$, $n$ being odd has no even divisors, so it follows that $e(n)=0$ and 
$o(n)=o(r)=(\alpha_1+1)\cdots(\alpha_r+1)=d(r)$.
If $\alpha\ge 1$, the number of odd divisors is still being equal to $d(r)$, while for each odd 
divisor of $n$ it corresponds $\alpha$ even divisors (those obtained by multiplying it by $2, 
2^2,\dots,2^\alpha$). Thus, we have the general formulas
\begin{equation}\label{eqoe}
   \begin{split}
   \begin{cases}
    o(n)=&d(r)\\ 
     e(n)=&\alpha\cdot d(r)
   \end{cases}
   ,\quad\text{for $n=2^\alpha r$, with $r$ odd.}
   \end{split}
\end{equation}
Two special cases, in which $e(n)$ attains its minimum and its maximum values occur. The minimum of
$e(n)$ appears often, since any odd $n$ has no even divisors. Thus $e(n)=0$ and $o(n)=d(n)$  for 
$n$ odd.
In this case a local maximum, $o(n)=2^k$, is attained if $n=1\cdot 3\cdot 5\cdots (2k+1)$.
At the other end, if $n$ is a power of $2$, then one is the only odd divisor of $n$, so
$e(n)=\alpha$ and $o(n)=1$  if $n=2^\alpha$.

\begin{figure}[ht]
 \centering
 \mbox{
 \subfigure{
    \includegraphics[width=0.48\textwidth]{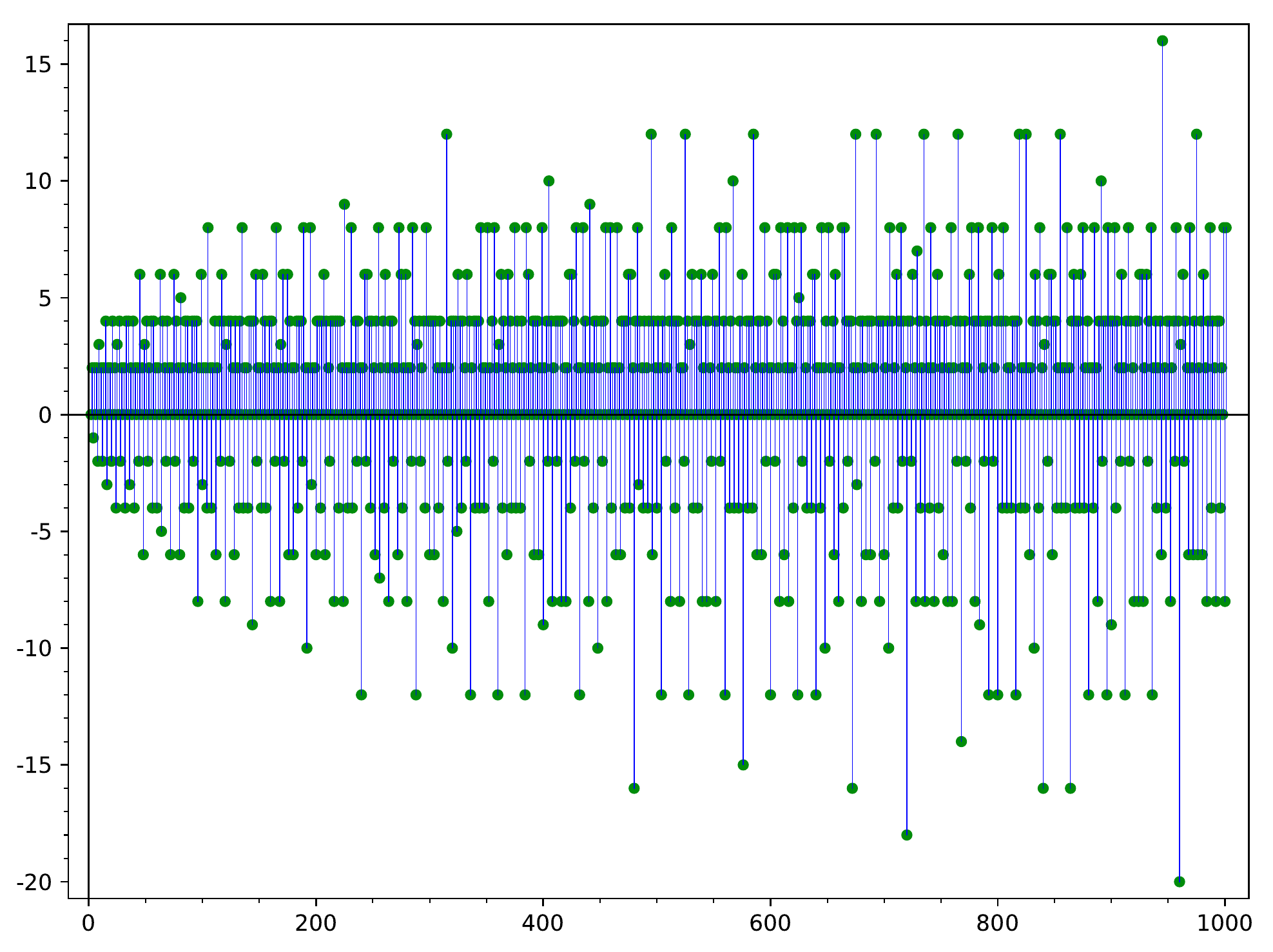}
 }
 \subfigure{
    \includegraphics[width=0.48\textwidth]{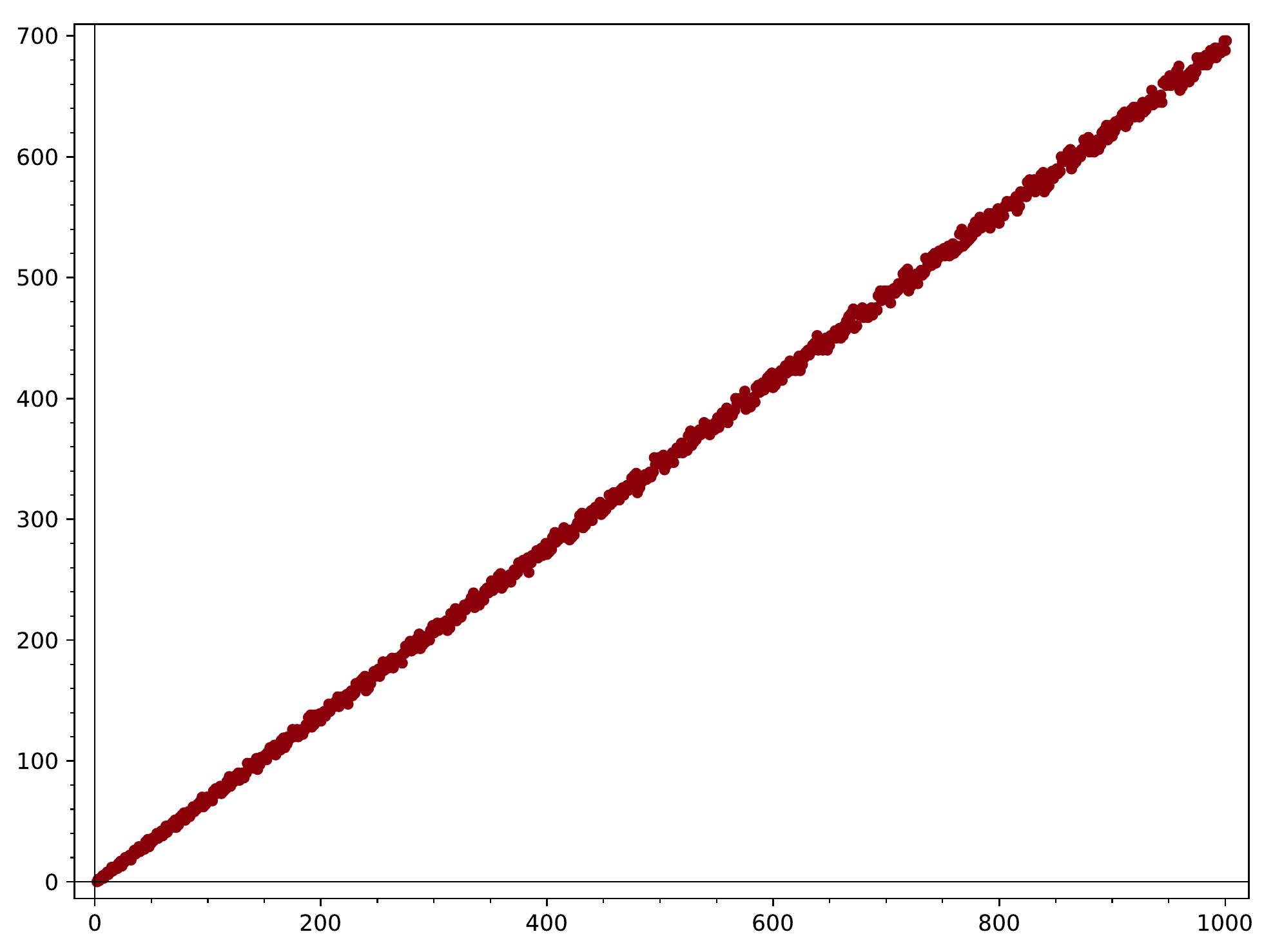}
 }
 }
\caption{The spread of 	$D(n)$ and its partial averages in the interval $[2, 1000]$.	
	}
 \label{Fig2-1000}
 \end{figure}

\begin{figure}[ht]
 \centering
 \mbox{
 \subfigure{
    \includegraphics[width=0.48\textwidth]{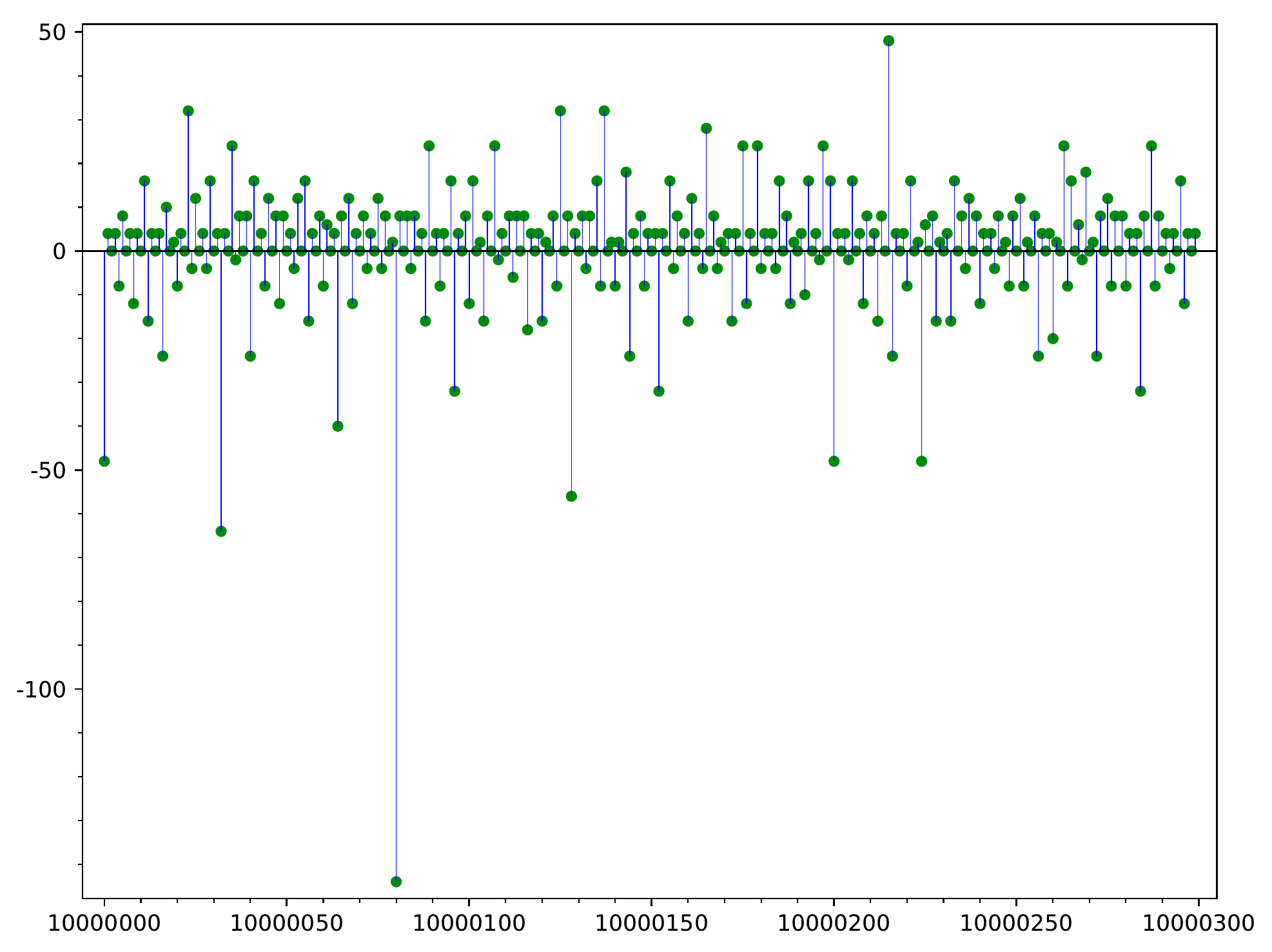}
 }
 \subfigure{
    \includegraphics[width=0.48\textwidth]{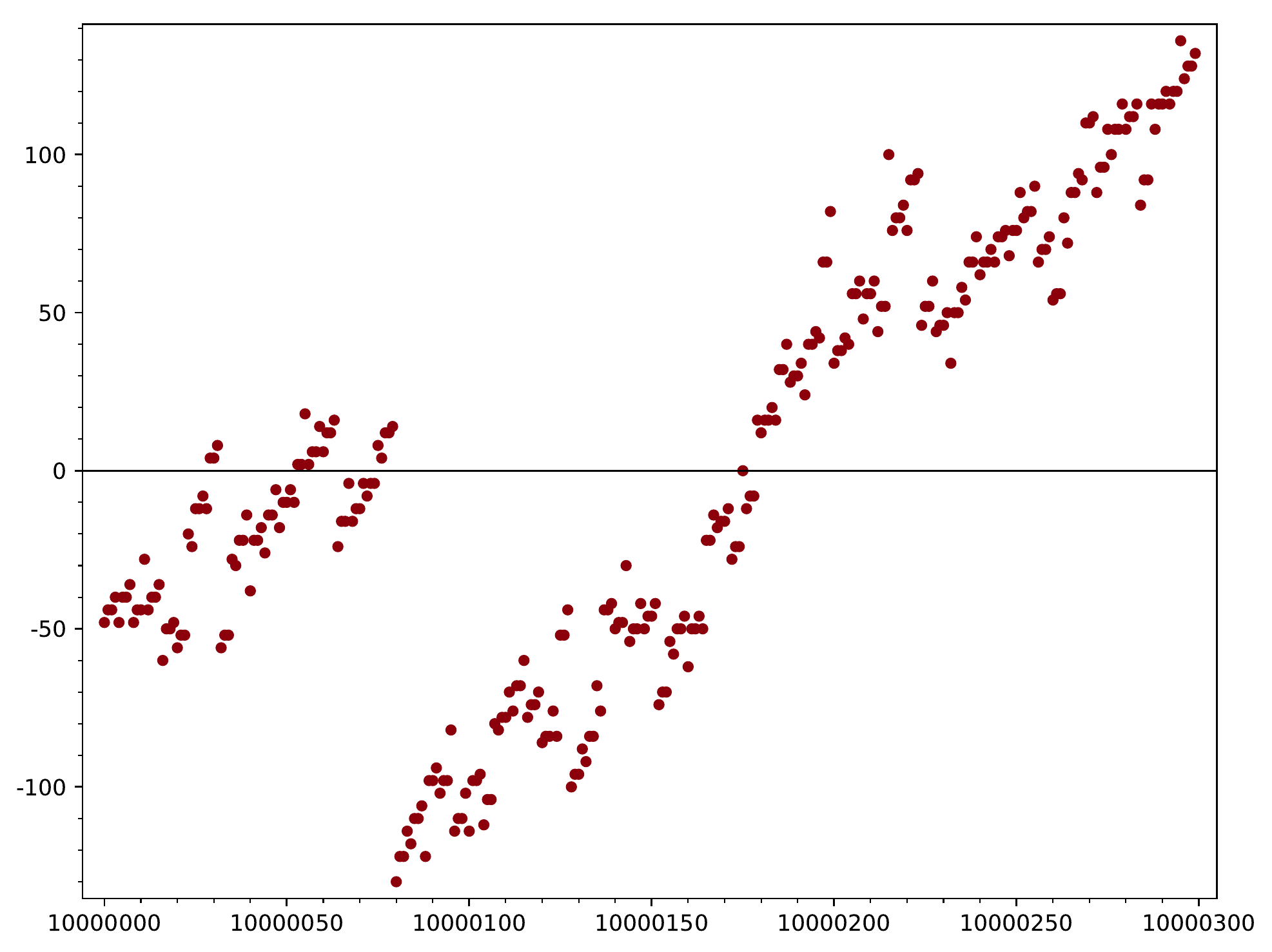}
 }
 }
\caption{Locally, in short intervals, in which some abnormal values of $D(n)$ occur, the partial 
averages 
may show a noisier effect.  This happens, for example, in the interval of length $300$ that starts 
at $10^7$.
}
 \label{Fig10la7p300}
 \end{figure}

Experimental results show that the partial averages of $D(n)$, that is, the sums 
$\sum_{n=1}^x D(n)$ 
tend to increase linearly with $x$ (see Figure~\ref{Fig2-1000}) and the same behavior is apparent 
for the averages calculated on shorter intervals. Although, in rare situations, the 
increase may be bumpy in shorter intervals that contain special $n$'s. For example, the larger jump 
in Figure~\ref{Fig10la7p300} is caused by the reach in divisors number
$n=10000080= 2^4 \cdot 3^2 \cdot 5 \cdot 17 \cdot 19 \cdot 43$, which implies $o(n)=48$, $e(n)=192$ 
and $D(n)=48-192=-144$.

Next, by \eqref{eqoe} the formula for the difference function is
\begin{equation}\label{eqDoe}
   \begin{split}
   D(n)=o(n)-e(n)=(1-\alpha)d(r), \quad\text{for $n=2^\alpha r$, $\alpha\ge 0$, $r$ odd.}
   \end{split}
\end{equation}
Notice that if $n$ is even but not divisible by four, than $e(n)=o(n)=d(n)$, so 
$D(n)=0$.

In the following, we assume $x$ is sufficiently large and calculate the average of $D(n)$ 
over the positive integers $n\le x$. Using formula \eqref{eqDoe}, we have:
\begin{equation}\label{eqDoe1}
   \begin{split}
   \sum_{n=1}^x D(n) & = 
   \sum_{\substack{r=1\\ r\text{ odd}}}^x 
   \sum_{\substack{\alpha=0\\ 2^\alpha r\le x}}^{\left[\frac{\log x}{\log 2}\right]} 
	    (1-\alpha)d(r)\,.
   \end{split}
\end{equation}
Let us denote the divisors sum over the odd integers by $I(t)$, that is 
\begin{equation*}
   \begin{split}
   I(t) := \sum_{\substack{r=1\\ r\text{ odd}}}^t d(r)\,.
   \end{split}
\end{equation*}
Then, \eqref{eqDoe1} becomes
\begin{equation}\label{eqsumI}
   \begin{split}
   \sum_{n=1}^x D(n) & = I(x)-I(x/4)-2I(x/8)-\cdots-(1-\tau)I(x/2^\tau) + R(x)\,,
   \end{split}
\end{equation}
where $\tau = \left[\frac{\log x}{\log 2}\right]$ and 
\begin{equation}\label{eqR}
   \begin{split} 
   R(x)= \tau I(y),
   \end{split}
\end{equation}
for some $y\le x/2^{\tau+1}$, collects the remaining terms.

\smallskip

We need the following two lemmas.
\begin{lemma}\label{LemmaArmonicOdds}
 We have
 \begin{equation*}
   \begin{split}
   A_o(x):=\sum_{\substack{1\le n\le x\\ n\ \mathrm{ odd}}}^x\frac 1n = \frac 12\log x+\frac{\log 
2}{2}+\frac{\gamma}{2}+O\left(\frac 1x\right),
   \end{split}
\end{equation*}
where $\gamma$ is Euler's constant.
\end{lemma}
\begin{proof}
 The harmonic sum up to $x$ is
 \begin{equation*}
   \begin{split}
   A(x)=\sum_{1\le n\le x}\frac 1n = \log x+\gamma+O\left(\frac 1x\right).
   \end{split}
\end{equation*}
Then 
\begin{equation*}
   \begin{split}
   A_o(x) & = A(x)-A(x/2)/2\\	
	    & = \log x+\gamma - \frac 12\log\frac x2 -\frac\gamma 2 + 
O\left(\frac 1x\right)\\
& = \frac 12\log x+\frac{\log 2}{2}+\frac{\gamma}{2}+O\left(\frac 1x\right),
   \end{split}
\end{equation*}
what had to be proved.
\end{proof}

The sum of odd divisors can be calculated by the well-known inclusion-exclusion Dirichlet method 
and this is the object of the next lemma.

\begin{lemma}\label{LemmaDirichlet}
 We have
 \begin{equation*}
   \begin{split}
   I(x)=  \sum_{\substack{n=1\\ n\ \mathrm{ odd}}}^x d(n) =
   \frac 14x\log x+x\left(\frac{\log 2}{2}+\frac\gamma 2 -\frac 14\right)+O\left(\sqrt x\right).
   \end{split}
\end{equation*}
\end{lemma}
\begin{proof}
 First we write the sum of the divisors as a double sum that counts lattice points under a 
hyperbola:
 \begin{equation*}
   \begin{split}
   I(x)=\sum_{\substack{n=1\\ n\ \mathrm{ odd}}}^x d(n) 
   = \sum_{\substack{1\le ab\le x\\ a,b\ \mathrm{ odd}}}^x 1\,.
   \end{split}
\end{equation*}
The contribution of the numerous smaller terms can be controlled efficiently counting them twice, 
in different order. Thus,  we have 
 \begin{equation*}
   \begin{split}
   I(x)&=\sum_{\substack{1\le a\le\sqrt x\\ a\ \mathrm{ odd}}} \
	  \sum_{\substack{1\le b\le\frac xa\\ b\ \mathrm{ odd}}} 1 +
	\sum_{\substack{1\le b\le\sqrt x\\ b\ \mathrm{ odd}}} \
	  \sum_{\substack{1\le a\le\frac xb\\ a\ \mathrm{ odd}}} 1 -
	  \sum_{\substack{1\le a\le\sqrt x\\ a\ \mathrm{ odd}}}\
	  \sum_{\substack{1\le b\le\sqrt x\\ b\ \mathrm{ odd}}}  1\\
    & = 2\sum_{\substack{1\le a\le\sqrt x\\ a\ \mathrm{ odd}}} 
	\left(\frac{x}{2a} +O(1)\right)
	-\left(\frac{\sqrt{x}}{2}+O(1)\right)^2\\
    & = xA_o(\sqrt{x})-\frac{x}{4} + O\left(\sqrt x\right).	
   \end{split}
\end{equation*}
  Then, using Lemma~\ref{LemmaArmonicOdds}, we find that
  \begin{equation*}
   \begin{split}
   I(x)&= x\left(\frac 12\log \sqrt x+\frac{\log 2}{2}+\frac{\gamma}{2}+O\Big(\frac 
1{\sqrt{x}}\Big)\right)
	  -\frac{x}{4} + O\left(\sqrt x\right)\\
	&= \frac 12x  \log \sqrt x +x\left(\frac{\log 2}{2}+\frac{\gamma}{2}-\frac 
14\right)+O\left(\sqrt x\right),
   \end{split}
\end{equation*}
which concludes the proof of the lemma.
\end{proof}
Now we are ready to obtain the estimate of the average of the difference of parity functions.

\begin{proposition}\label{Proposition1}
 We have
 \begin{equation*}
   \begin{split}
   \sum_{1\le n\le x}D(n)= \log 2 \cdot x +O\big(\sqrt x\big)\,.
   \end{split}
\end{equation*}
\end{proposition}
\begin{proof}
 Replacing the corresponding terms of the sum in~\eqref{eqsumI} and in the relation \eqref{eqR} by 
their estimate from Lemma~\ref{LemmaDirichlet}, we have
\begin{equation}\label{eqsumII}
   \begin{split}
   \sum_{n=1}^x D(n) & = I(x)-I(x/4)-2I(x/8)-\cdots-(1-\tau)I(x/2^\tau) + R(x)\\
   &=\tfrac 14 x\log x S_1(\tau)+\tfrac{\log 2}{4}xS_2(\tau) + 
   \tfrac{2\log 2+2\gamma-1}{4}xS_1(\tau)+
   O\big(\sqrt x S_3(\tau)\big) +R(x)\,,  
   \end{split}
\end{equation}
where we denoted
\begin{equation*}
   \begin{split}
   S_1(\tau)&=1-\frac{1}{2^2}-\frac{2}{2^3}-\cdots-\frac{\tau-1}{2^\tau},\\
   S_2(\tau)&=\frac{1\cdot 2}{2^2}+\frac{2\cdot 3}{2^3}+\cdots+\frac{(\tau-1)\tau}{2^\tau},\\
   S_3(\tau)&=1-\frac{1}{2^{2/2}}-\frac{2}{2^{3/2}}-\cdots-\frac{\tau-1}{2^{\tau/2}}
   \end{split}
\end{equation*}
and $\tau = \left[\frac{\log x}{\log 2}\right]$.
The sums $S_1(\tau), S_2(\tau), S_3(\tau)$ can be added and expressed in closed-form and then their 
sizes can be easily evaluated. Thus we find that all terms except the second from the right hand 
side of \eqref{eqsumII} are no larger than $O(\sqrt x)$.
Then, since 
\[
 S_2(\tau)=4+O\left(\frac{\log x}{x}\right),
\]
the main term on the right hand side of \eqref{eqsumII} is
\[
 \frac{\log 2}{4}xS_2(\tau) = x\log 2+O(\log x),
\]
which concludes the proof.
\end{proof}
Notice that the experiment drawn in Figure~\ref{Fig2-1000} is confirmed by the slope from the above 
proposition since $\log 2\approx 0.69314$.

\section{Infinite words and Dirichlet series}\label{SectionWordsAndDirichletSeries}

Let $A$ be a finite alphabet. Given a map $H : A\rightarrow\mathbb{C}$ and an infinite
word $w:\mathbb{N}\rightarrow A$, we compose $H$ with $w$ and consider the associated
Dirichlet series
\begin{equation}\label{e1}
F(H, w, s):= \sum_{n=1}^{\infty} \frac{H(w(n))}{n^s}\,,
\end{equation}
which is absolutely convergent in the half-plane $\Re s > 1$.  The analytic function 
$F(H, w, s)$ captures some properties of the given word $w$. If $H$ is injective, the
word is uniquely determined by the function  $F(H, w,  s)$. More precisely,
the coefficients $H(w(n))$ can be recovered from $F(H, w, s)$ via Perron type formulas.
For any positive integer $n$, any $x$ in the interval $(n, n+1)$, and any real number
$c>1$,
\begin{equation}\label{e2}
\sum_{j=1}^n H(w(j)) =\frac1{2\pi i}\int_{c-i\infty}^{c+i\infty} \frac{F(H, w, s) x^s}{s}\, ds\,,
\end{equation}
where the integral above is to be interpreted as the symmetric limit 
$\lim_{T\rightarrow\infty}\int_{c-iT}^{c+iT}\,.$

Working with the Dirichlet series $F(H, w, s)$ in the half-plane of convergence $\Re s > 1$
may reveal various properties of the given words $w$. More interesting are cases when the functions $F(H, w, s)$ have analytic or meromorphic continuation to larger half-planes.
For example, for a word of the form $w = aaaa\dots$, we have
\begin{equation}\label{e3}
F(H, w, s) = \sum_{n=1}^{\infty} \frac{H(a)}{n^s} = H(a) \zeta(s),
\end{equation}
a constant multiple of the Riemann zeta function $\zeta(s)$. In this case $F(H, w, s)$
has a meromorphic continuation to the entire complex plane, the only pole being 
a simple pole at $s = 1$ 
(for $H(a)$ nonzero). Similarly, any Dirichlet $L-$function, and more generally any finite
linear combination of Dirichlet $L-$functions $c_1 L(s, \chi_1)+\cdots+ c_k L(s, \chi_k)$
is of the form
\begin{equation}\label{e4}
c_1 L(s, \chi_1)+\cdots+ c_k L(s, \chi_k) = F(H, w, s),
\end{equation}
for some alphabet $A$, some map $H : A\rightarrow\mathbb{C}$, and some 
periodic word $w$. In all these cases $F(H, w, s)$ has analytic continuation to the entire
complex plane, with a possible pole at $s=1$ if one or more of the characters
$\chi_1,\dots, \chi_k$ are principal.

Let now $A$ be an alphabet, $H : A\rightarrow\mathbb{C}$, and let $\lambda$ be
a real number satisfying $0\le \lambda <1$. Suppose $w$ and $w'$ are two words
that coincide at enough many places so that for any $\varepsilon >0$,
\begin{equation}\label{e5}
\lim_{n\rightarrow\infty}\frac{|\{j: 1\le j\le n, w'(j) \ne w{j}\}|}{n^{\lambda+\varepsilon}}
= 0\,.
\end{equation}

Then it is easy to see that the difference $F(H, w', s) - F(H, w, s)$ has analytic
continuation to the half-plane $\Re s > \lambda$. Therefore in such cases
$F(H, w', s)$ has analytic (respectively meromorphic) continuation to the half-plane 
$\Re s > \lambda$ if and only if $F(H, w, s)$ has the same property.

\section{Dirichlet series associated to Sturmian words}\label{SectionDirichletSeries}

A one-parameter family of Dirichlet series whose coefficients are Sturmian words is studied by 
Kwon~\cite{K2015}. For our purpose, we have proceed as follows.
If $w$ is a Sturmian word, $F(H, w, s)$ has
meromorphic continuation to the half-plane $\Re s > 0$, the only
pole being a simple pole at $s=1$. Therefore
for any word $w'$ which coincides with a Sturmian word $w$ at enough many
positions so that \eqref{e5} holds for some $0\le \lambda < 1$, the 
corresponding Dirichlet series $F(H, w', s)$ has meromorphic continuation
to the half-plane $\Re s > \lambda$,  the only
pole being a simple pole at $s=1$.

Let $w$ be a Sturmian word over the alphabet $A = \{ a, b\}$, and assume for 
simplicity that $H(a) = 0$ and $H(b) = 1$. The proportion of positions $j$ up to $n$
for which $w(j) = b$ has a limit as $n$ tends to infinity. Let $\beta_w$ denote this limit,
\begin{equation}\label{e6}
\beta_w:= \lim_{n\rightarrow \infty} \frac {|\{ 1\le j\le n :  w(j) = b\}|}{n}
=\lim_{n\rightarrow \infty} \frac 1n\sum_{1\le j\le n} H(w(j))\,.
\end{equation}

Any two factors of $w$ having the same length contain about the same number of
$a$'s and~$b$'s. To be precise, for any positive integers $n_1$, $n_2$ and $L$,
\begin{equation}\label{e7}
\left| \sum_{n_1\le j < n_1+L} H(w(j)) -  \sum_{n_2\le j < n_2+L} H(w(j))\right| \le 1.
\end{equation}
Therefore, for any positive integer $n$,
\begin{equation}\label{e8}
\left| \beta_w n -  \sum_{1\le j \le n} H(w(j))\right| \le 1.
\end{equation}

For $\Re s>1$, we rewrite $F(H, w, s)$ as 
\begin{equation}\label{e9}
F(H, w, s) = \sum_{n=1}^{\infty} \frac{H(w(n))}{n^s} = 
\sum_{n=1}^{\infty} H(w(n)) \int_n^{\infty} \frac s{t^{s+1}}\,dt.
\end{equation}

We further have
\begin{equation}\label{e10}
F(H, w, s) = s\int_1^{\infty} \frac {\sum_{n\le t} H(w(n)) }{t^{s+1}}\,dt,
\end{equation}
which may be rewritten as
\begin{equation}\label{e11}
F(H, w, s) =  \frac{\beta_w}{s-1} + \beta_w
-s\int_1^{\infty} \frac {\beta_w t - \sum_{n\le t} H(w(n)) }{t^{s+1}} dt.
\end{equation}

By \eqref{e8}, the numerator under the integral on the right side of \eqref{e11}
is $O(1)$. It follows that this integral represents an analytic function of $s$ in the
entire half-plane $\Re s > 0$. In conclusion, $F (H, w, s)$ has an analytic
continuation to the half-plane $\Re s > 0$, with the exception of a simple pole
at $s=1$, with residue $\beta_w$.

In order to take advantage of the pole at $s=1$ in concrete applications, we need
to have some knowledge on the growth of $|F(H, w, s)|$ for $s = \sigma +it$ with $|t|$
large and $\sigma$ not too close to zero.  For any fixed $\delta>0$, one has
\begin{equation}\label{e12}
\left| F(H, w, \sigma+it)\right| = O_{\delta} \left(|t| ^{1- \sigma+\delta}\right),
\end{equation}
uniformly for all $\delta \le \sigma \le 1+ \delta$ and $|t| \ge 1$. 
For points on the vertical line
$\sigma = 1 + \delta$, \eqref{e12} follows directly from \eqref{e1}.
For points on the vertical line
$\sigma = \delta$, \eqref{e12} follows from \eqref{e11}, taking into
account that the numerator under the integral in \eqref{e11}
is $O(1)$. Then the convexity bound \eqref{e12}, for all  
$\delta \le \sigma \le 1 + \delta$, follows from the general
theory of Dirichlet series (see Titchmarsh~\cite{T1939}).

\section{Proof of Theorem~\ref{Theorem1}}\label{SectionProof}

To prove the sharp asymptotic formula~\eqref{eq24}, we will make use of the properties of  $F(H, w, 
s)$ discussed
above.

Recall that the Dirichlet convolution of two arithmetical functions 
$f, g: \mathbb{N}\rightarrow\mathbb{C}$ is the function 
$f*g: \mathbb{N}\rightarrow\mathbb{C}$ defined by
\begin{equation}\label{e15}
(f*g)(n) = \sum_{d | n} f(d) g\left(\frac nd\right).
\end{equation}

Let us observe that $D$ is the Dirichlet convolution of $H\circ w$ with the function
$h$ given by $h(n) = (-1)^{n+1}$. Indeed,
\begin{equation}\label{e16}
((H\circ w)*h)(n) = \sum_{d | n} H(w(d)) h\left(\frac nd\right) = 
\sum_{\substack{d | n\\ w(d) = b}} (-1)^{\frac nd +1} = D_w(n)\,.
\end{equation}

Dirichlet convolution corresponds to multiplication of the associated
Dirichlet series. Therefore, in the half-plane of absolute convergence $\Re s > 1$, 
\begin{equation}\label{e17}
\sum_{n=1}^{\infty} \frac{D_w(n)}{n^s} = F(H, w, s) \sum_{n=1}^{\infty} \frac{(-1)^{n+1}}{n^s}\,.
\end{equation}
Here the sum on the right side of \eqref{e17} equals
\begin{equation}\label{e18}
1-\frac1{2^s}+\frac1{3^s} - \frac1{4^s} +\dots = \left(1-\frac2{2^s}\right)
\sum_{n=1}^{\infty} \frac 1{n^s}\,,
\end{equation}
and we obtain
\begin{equation}\label{e19}
\sum_{n=1}^{\infty} \frac{D_w(n)}{n^s} = \left(1-\frac1{2^{s-1}}\right) \zeta(s)
F(H, w, s) \,.
\end{equation}

Next, we use a variant of Perron's formula~\cite{T1986} in combination with \eqref{e14} and  
\eqref{e19}
in order to express $M_w(x)$ as an integral over a vertical line. For any real numbers
$x\ge 1$ and $c>1$,
\begin{equation}\label{e20}
M_w(x) = \sum_{n\le x} D_w(n) \left(1-\frac nx\right) = 
\frac1{2\pi i}\int_{c-i\infty}^{c+i\infty} \frac{ \left(1-\frac1{2^{s-1}}\right) \zeta(s)
F(H, w, s) x^s}{s(s+1)}\, ds\,.
\end{equation}

The integrand on the right side of \eqref{e20} is analytic on the entire half-plane
$\Re s > 0$, except for a pole at $s=1$.  Notice that at $s=1$ both $\zeta(s)$ and
$F(H, w, s)$ have simple poles, while $1- 1/2^{s-1}$ has a simple zero. Therefore
the integrand on the right side of (20) has a simple pole at $s=1$. The Taylor series
expansion of $1- 1/2^{s-1}$ about $s = 1$ is
\begin{equation}\label{e21}
1-\frac1{2^{s-1}} = 1 -e^{-(s-1)\log 2} = (s-1)\log 2 +\dots
\end{equation}

Also, as we know,
\begin{equation}\label{e22}
\zeta(s) = \frac{1}{s-1} +  \text{analytic}
\end{equation}
and
\begin{equation}\label{e23}
F(H, w, s) = \frac{\beta_w}{s-1} +  \text{analytic}.
\end{equation}

Using  \eqref{e21}, \eqref{e22} and \eqref{e23} it follows that the residue at $s=1$ 
of the integrand on the right side of \eqref{e20} equals
$\frac{\beta_w\log 2}{2}\, x.$

Let now $T>1$  be a parameter, whose precise value will be given later.
We fix a small $\delta >0$, and then shift the line of integration on the right side 
of \eqref{e20} to the left. In doing so, we encounter the pole at $s=1$.
We choose the new contour as follows.
We start vertically from $1+\delta-i\infty$ to $1+\delta -iT$, then move left to $\delta - iT$,
then move vertically to $\delta+iT$, then go horizontally to $1+\delta +iT$, and then
vertically to $1+\delta+i\infty$.  Next, employing \eqref{e12} in combination with
known bounds for $\zeta(s)$ on the critical strip, we find
that if we choose~$T = x^{2/3}$, then the integral over the new contour is bounded as
$O_\delta\left(x^{\frac13+\delta}\right)$. Lastly, by the residue theorem we obtain
the desired asymptotic formula,
which completes the proof of the theorem.

\bigskip
\bigskip

\bigskip
\bigskip

\textit{Acknowledgements. } {\small Calculations and plots created using the free open-source 
mathematics
software system \texttt{SAGE}: 
\href{http://www.sagemath.org}{http://www.sagemath.org}}.


\end{document}